\documentclass{scrartcl}
\usepackage{amsmath, amssymb, amsthm, amscd}
\usepackage{tikz}
\usetikzlibrary{patterns}

\newcommand{\pk}{\Phi^+(K)}
\renewcommand{\tt}{\mathfrak{t}}
\renewcommand{\gg}{\mathfrak{g}}
\renewcommand{\ss}{\mathfrak{s}}
\newcommand{\cc}{\mathbb{C}}
\newcommand{\kk}{\mathfrak{k}}
\newcommand{\dd}{\overline{\partial}}
\newcommand{\ttr}{\mathfrak{t}^*_{reg}}
\newcommand{\tpim}{T_{p_i}M}
\newcommand{\tpif}{T_{p_i}F}
\newcommand{\npif}{(NF)_{p_i}}

\newcommand{\cit}{c_{1,T}}

\newcommand{\tdt}{Td_T}
\newcommand{\pgk}{\Phi^+(G / K)}

\newcommand{\bi}{\beta_i}

\newcommand{\pg}{\Phi^+(G)}

\renewcommand{\ll}{L_\lambda}
\newcommand{\bb}{\mathcal{B}}
\newcommand{\qm}{Q(M, \omega, \nabla, L; J)}
\newcommand{\wg}{\mathcal{W}_G}
\DeclareMathOperator{\Imm}{im}
\DeclareMathOperator{\coker}{coker}

\newtheorem{thm}{Theorem}[section]

\newtheorem{lem}[thm]{Lemma}
\newtheorem{cor}[thm]{Corollary}

\theoremstyle{definition}
\newtheorem{defn}[thm]{Definition}
\newtheorem{example}[thm]{Example}

\theoremstyle{remark}
\newtheorem*{remark}{Remark}

\begin{document}
\title{\boldmath Character and Multiplicity Formulas for Compact Hamiltonian $G$-spaces}
\author{Elisheva Adina Gamse\thanks{Partially supported by NSF grant DMS 12--11819} \\ Department of Mathematics, Northeastern University\\ 360 Huntingdon Avenue, Boston MA 02115\\\texttt{gamse.e@husky.neu.edu}}

\maketitle

\begin{abstract}
Let $K \subset G$ be compact connected Lie groups with common maximal torus $T$. Let $(M, \omega)$ be a prequantisable compact connected symplectic manifold with a Hamiltonian $G$-action. Geometric quantisation gives a virtual representation of $G$; we give a formula for the character $\chi$ of this virtual representation as a quotient of virtual characters of $K$. When $M$ is a generic coadjoint orbit our formula agrees with the Gross-Kostant-Ramond-Sternberg formula. We then derive a generalisation of the Guillemin-Prato multiplicity formula which, for $\lambda$ a dominant integral weight of $K$, gives the multiplicity in $\chi$ of the irreducible representation of $K$ of highest weight $\lambda$.
\end{abstract}

\tableofcontents

\section{Introduction}

Let $(M,\omega)$ be a compact connected symplectic manifold, and let $G$ be a compact connected Lie group acting on $M$ in a Hamiltonian fashion with moment map $\mu:M \to \gg^*$. Assume that the equivariant cohomology class $[\omega + \mu]$ is integral. Let $L\to M$ be a complex Hermitian line bundle and let $\nabla$ be a Hermitian connection on $L$ whose equivariant curvature form is the equivariant symplectic form $\omega + \mu$. Suppose that the $G$-action on $M$ lifts to a $G$-action on $L$ which preserves $\nabla$. Let $J$ be a $G$-equivariant almost complex structure on $M$. This almost complex structure gives us a totally complex distribution $\Delta$ in the complexified tangent bundle to $M$, such that $TM\otimes \cc= \Delta \oplus \overline{\Delta}$. This gives us a splitting $\Lambda^k(TM\otimes \cc) = \sum_{i+j=k}\Lambda^i(\overline{\Delta})\otimes\Lambda^j(\Delta)$, and thus a bigrading on the space of $L$-valued differential forms. We will take $J$ to be compatible with $\omega$; that is, $\omega$ is a $(1,1)$-form and $\omega(v,Jv)>0$ for all $p \in M$ and $v \in T_pM$. 

Let $D:\Omega^k(M;L)\to \Omega^{k+1}(M;L)$ be the operator $D(s\otimes \alpha)=\nabla s \otimes \alpha + s \otimes d\alpha$, where $s \in \Gamma(L)$ and $\alpha \in \Omega^k(M)$. Define $\overline{\partial}$ to be the $(0,k+1)$ component of $D$. Fix a Hermitian metric on $M$. Together with the Hermitian inner product on $L$, this gives a Hermitian inner product on $L \otimes \Lambda^{0,k}TM$. Define operators $\overline{\partial}^t:\Omega^{0,k}(M;L)\to \Omega^{0,k-1}(M;L)$ which are $\mathcal{L}^2$-adjoint to $\dd$. Then the operator $\dd + \dd^t:\Omega^{0,even}(M;L)\to \Omega^{0,odd}(M;L)$ is elliptic; the \emph{quantisation} $\qm$ is defined as the virtual $G$-representation on $\ker(\dd+\dd^t)\ominus \coker(\dd+\dd^t)$; that is, the equivariant index of the $\dd+\dd^t$ operator on $L$. (See chapter 6 of Ginzburg, Guillemin and Karshon's book \cite{ggk}.)

In their paper \cite{gp}, Guillemin and Prato prove a formula for the multiplicity with which each irreducible character of $G$ appears in the character of this representation $Q(M, \omega, L, \nabla; J)$. In the special case of a torus $T$ acting on a coadjoint orbit $G/T$ by left multiplication, their formula becomes the Kostant multiplicity formula that Kostant obtained in \cite{kostant}.

Let $G$ be semisimple, let $K \subset G$ be a Lie subgroup of equal rank, choose a common maximal torus $T\subset K \subset G$, let $\lambda$ be a dominant integral weight for $G$, and let $M$ be the coadjoint orbit $G\cdot \lambda$. The choice of positive roots for $G$ determines a complex structure on $G\cdot \lambda$; take $J$ to be the corresponding almost complex structure on $G\cdot \lambda$. Let $L_\lambda=G\times_T \cc_{(\lambda)}$ (where $T$ acts on the line $\cc_{(\lambda)}$ with weight $\lambda$). Then the quantisation $Q(M, \omega, \ll, \nabla; J)$ is a $G$-representation on the space of holomorphic sections of $\ll$ (see \cite{ggk} for details). The Borel-Weil theorem tells us that the quantisation of $(M, \omega, \ll, \nabla, J)$ is the irreducible representation of $G$ of highest weight $\lambda$, and that all irreducible representations arise in this way, as described by Bott in \cite{bott}. In this case, Gross, Kostant, Ramond and Sternberg provided in \cite{gkrs} a formula for the character of this $G$-representation as a quotient of the alternating sum of a multiplet of $K$-characters. Their formula has its origins in String Theory and is the motivation for our work which provides a generalisation. In the special case where $K=T$, their formula becomes the Weyl character formula. 

In this paper, we extend the result of Gross, Kostant, Ramond and Sternberg by replacing the coadjoint orbit $G\cdot \lambda$ with any compact connected symplectic Hamiltonian $G$-manifold $M$, and relate the resulting character formula to the Guillemin-Prato multiplicity formula. In section two we obtain, for arbitrary compact connected symplectic manifolds $M$ with Hamiltonian $G$-actions, a formula for the character of $\qm$ as a quotient of $K$-characters. In section three, we derive from our formula a generalisation of the Guillemin-Prato multiplicity formula. 

\paragraph{Acknowledgements} I would like to thank Jonathan Weitsman for suggesting the problem and for his advice and encouragement throughout. I am also grateful to the anonymous reviewer for making valuable corrections and suggestions, and to Nate Bade, Barbara Bolognese, Victor Guillemin, and Ryan Mickler for helpful discussions.

\section{Character Formula}

Let $G$ and $K$ be compact connected Lie groups of equal rank with $K \subset G$, and choose a common maximal torus $T\subset K \subset G$. Write $\tt$, $\kk$ and $\gg$ for the Lie algebras of $T$, $K$, and $G$ respectively. Let $\mathcal{N}_G(T)$ denote the normaliser in $G$ of $T$ and let $W(G) = \mathcal{N}_G(T)/T$ be the Weyl group of $G$. Choose a set $\pg$ of positive roots of $G$, and let $\wg \subset \tt^*$ be the positive Weyl chamber for $G$. Let $\Lambda \subset \tt^*$ denote the weight lattice. For $\phi \in \pg$, let $H_\phi$ denote the hyperplane orthogonal to $\phi$ in $\tt^*$, and let $w_\phi \in W(G)$ be the reflection in this hyperplane. Let $(M,\omega)$ be a compact connected symplectic manifold, and let $G$ act on $(M, \omega)$ in a Hamiltonian manner. Then $T$ acts on $(M, \omega)$ in a Hamiltonian fashion with $T$-equivariant moment map $\mu: M \rightarrow \tt^*$; we will assume that the fixed points of this torus action are isolated. Suppose the equivariant cohomology class $[\omega + \mu]$ is integral, choose a prequantisation line bundle $(L,\nabla)$, and let $J$ be a $G$-equivariant almost complex structure on $M$ that is compatible with $\omega$. Let $\qm = ind (\dd+\dd^t)$ be the quantisation of $(M, \omega, L, \nabla; J)$, and let $\chi$ denote its character. In this section we will give an expression for the character $\chi$, as a sum of quotients of (virtual) $K$-characters. We begin by setting up the equivariant cohomology we need, and by recalling the equivariant index theorem and the localisation theorem which will be our main tools.

\subsection{Review of Equivariant Cohomology}

Let $G$ be a compact Lie group acting on a manifold $M$. Let $EG$ be a contractible space on which $G$ acts freely, so that $M\times EG \simeq M$ and the diagonal action of $G$ on $M \times EG$ is free. Form the homotopy quotient $M_G := (M\times EG)/G$. 

\begin{defn} The \emph{equivariant cohomology ring} $H^*_G(M)$ is the ordinary cohomology ring $H^*(M_G)$. \end{defn}

Let $H\subset G$ be a subgroup. Then $H$ also acts freely on $EG$, so we can take $EH=EG$ and thus $M_G = M_H/G$. If $p:M_H\to M_G$ denotes the projection, we can pull back classes in $H_G^*(M)$ along $p$ to $H^*_H(M)$.

An alternative approach to defining equivariant cohomology is known as the Cartan model. Define an \emph{equivariant differential form} to be a $G$-equivariant polynomial on $\gg$ taking values in $\Omega^*(M)$. More precisely, the \emph{equivariant differential $k$-forms} are elements of $\Omega^k_G(M)=\bigoplus_{k=2i+j}(S^i(\gg^*)\otimes \Omega^j(M))^G$. The \emph{equivariant exterior differential} $d_G: \Omega^k_G(M) \rightarrow \Omega^{k+1}_G(M)$ is given by $$(d_G\alpha)(X) = d(\alpha(X)) - i_{X_M}\alpha(X),$$ where $\alpha \in \Omega^k_G(M)$, $X \in \gg$, and $X_M$ is the vector field defined by the infinitesimal action of $X$ on $M$. Note that $d_G^2=0$ by invariance.

Hence, an equivariant form $\alpha \in \Omega^*_G(M)$ is closed if $d(\alpha(X)) - i_{X_M}\alpha(X)=0$ for all $X \in \gg$.

\begin{example}[Equivariant symplectic form]
Let $\omega + \mu$ be the equivariant symplectic form. Then $d_G(\omega+\mu)(X) = d(\omega+ \mu)(X) - i_{X_M}(\omega + \mu)(X).$ But $d\omega=0$ since $\omega$ is the non-equivariant symplectic form, and $i_{X_M}(\mu)=0$ since $\mu$ is a 0-form. So $d_G(\omega+\mu)(X) = d\mu(X)-i_{X_M}\omega(X)$, which is zero by definition of the moment map $\mu$. So the equivariant symplectic form is equivariantly closed.
\end{example}

\begin{thm}[Equivariant de Rham Theorem] Let $G$ be a compact connected Lie group acting on a manifold $M$. Then the equivariant cohomology is given by $$H^*_G(M) = \frac{\ker d_G}{\Imm d_G}.$$ \end{thm}

Suppose $S \subset T$ is a subtorus with Lie algebra $\mathfrak{s}$. A $T$-equivariant form $\alpha$ is an $\Omega(M)$-valued polynomial on $\tt$; by restriction we can view this as an $\Omega(M)$-valued polynomial on $\ss$ and hence as an $S$-equivariant form. If $\alpha$ is $T$-equivariantly closed then $d(\alpha(X)) - i_{X_M}\alpha(X)=0$ for all $X \in \tt$, in which case we certainly have $d(\alpha(X)) - i_{X_M}\alpha(X)=0$ for all $X \in \ss$ and so $\alpha$ is also $S$-equivariantly closed.

Let $\cit$, $e_T$ and $\tdt$ denote the $T$-equivariant first Chern, Euler and Todd classes respectively. For a detailed discussion of these equivariant characteristic classes, see \cite{gs} or \cite{ggk}. 

We are now ready to state the following two theorems, which are the key components of our proof. For more details about these theorems we refer the reader to \cite{ggk}.

\begin{thm}[Equivariant Index Theorem] \label{eqind}Let $(M, \omega, L, \nabla; J)$ be a $G$-manifold, and let $\chi$ denote the character of the quantisation $\qm$. Then, in a neighbourhood of $0\in \gg$, $$\chi \circ exp = \int_M e^{\cit(L)}\tdt(M).$$  

\end{thm}

\begin{thm}[Atiyah-Bott-Berline-Vergne Localisation Formula] Let $M$ be a compact, oriented manifold and let the torus $T$ act on $M$ with fixed point set $F$, and let $\iota:F\rightarrow M$ denote the inclusion. Let $\alpha$ be an equivariantly closed form on $M$. Let $NF$ denote the normal bundle to $F$ in $M$. Then $$ \int_M \alpha = \int_F\frac{\iota^*_F \alpha}{e_T(NF)}.$$ In particular, when $F$ is a finite set of isolated points, $$\int_M \alpha = \sum_{p\in F} \frac{\alpha \vert_p}{e_T(T_pM)}.$$ \end{thm}

\subsection{Actions of the Torus and of a Subtorus}

The Equivariant Index Theorem and the Localisation Theorem together give an expression for the character of $\qm$ as a sum of contributions from each fixed point of the torus action on $M$. The fixed point data involved is the $T$-equivariant Chern classes of the line bundle $L$, and the equivariant Todd and Euler classes of the tangent bundle $TM$, all restricted to each fixed point $p$. These classes are all determined by the weights of the torus action on the fibres of the vector bundles above $p$. So we now discuss what the fixed points are and how the torus acts on the fibres above them.

\begin{lem} \label{image} The equivariant first Chern class $\cit(L)\vert_p$ of the fibre of the line bundle $L$ above the point $p \in M$ is the moment map image $\mu(p)$. \end{lem}

\begin{proof} We chose our prequantisation line bundle $L \rightarrow M$ and connection $\nabla$ such that its equivariant curvature form was equal to the equivariant symplectic form $\omega + \mu$. Restricted to a point $p$ this becomes the image $\mu(p)$.
\end{proof}

\begin{lem} \label{wactsonp} The Weyl group $W(G)$ of $G$ acts on the set of fixed points of the $T$ action on $M$. \end{lem}

\begin{proof} Let $p\in M$ be a fixed point of the torus action, and let $wT \in W(G)$. We define the action of $W(G)$ on the fixed point set by $wT \cdot p = w \cdot p$. Since the torus action fixes $p$, this doesn't depend on the coset representative $w$, and so the action is well-defined. We would like to show that $wT \cdot p$ is a fixed point of the torus action. By definition of the Weyl group, if $wT\in W(G)$ and $t \in T$, then $tw = wt'$ for some $t' \in T$. So $t\cdot (w \cdot p) = tw \cdot p = wt' \cdot p = w \cdot (t' \cdot p) = w \cdot p$, so $w \cdot p$ is a fixed point as claimed. \end{proof}

\begin{cor} The fixed points of the torus action can be partitioned into $W(G)$ orbits. \end{cor}

\begin{lem} \label{wactsonweights}Let $\alpha_1, \alpha_2, \ldots , \alpha_n$ be the weights of the representation of $T$ on the tangent space $T_pM$ to the fixed point $p$, and let $w \in \mathcal{N}_G(T)$ be a representative for $wT \in W(G)$. Then the weights of the $T$-representation on $T_{w\cdot p}M$ are $w\cdot \alpha_1, \ldots, w\cdot \alpha_n$.\end{lem}

\begin{proof} Let $\alpha$ be a weight of the torus action on the tangent space $T_pM$, and let $V_\alpha \subset T_pM$ be the weight space of $\alpha$. Let $v\in V_\alpha \subset T_pM$, and let $X \in \tt$ with $\exp X = t$. By definition this means $t \cdot v = e^{\langle \alpha, X \rangle}v$. We wish to know how $T$ acts on $w\cdot v$. Since $w \in \mathcal{N}_G(T)$, we know $w^{-1}tw = t'$ for some $t' \in T$. Suppose $t' = expY$. So $t\cdot (w \cdot v) = (tw)\cdot v = wt' \cdot v = w \cdot (t' \cdot v) = w \cdot (e^{\langle \alpha, Y \rangle}v) = e^{\langle \alpha, Y \rangle}w\cdot v $. But $\exp Y = w^{-1}\exp X w$, so $Y = w^{-1} X w$ since the exponential map commutes with the adjoint action. So $\langle \alpha, Y \rangle = \langle \alpha, w^{-1} X w \rangle = \langle w\cdot\alpha, X \rangle$, so $t\cdot w \cdot v = e^{\langle w \cdot\alpha, X \rangle}w\cdot v $. That is, $w\cdot \alpha$ is a weight of the $T$ action on the tangent space $T_{w\cdot p}M$, as required. \end{proof}

If we apply localisation to the equivariant index theorem, we get a formula for the character of the $G$-representation in terms of torus characters. We combine these torus characters into $K$-characters. Geometrically, this arises as follows.

Let $\ss \subseteq \tt$ be the algebra $\{\zeta \in \tt \vert \langle \alpha , \zeta \rangle = 0 \forall \alpha \in \Phi(K)\}$, and let $S = \exp(\ss)$. Then $S\subseteq T$ is the maximal torus of the centraliser of $K$ in $G$. 

\begin{lem} The fixed manifolds of the $S$-action are preserved by the $K$-action.
\end{lem}

\begin{proof} Suppose $x \in M$ is fixed by the $S$ action. We claim that $k\cdot x$ is also fixed by $S$, for all $k \in K$. But $S$ is contained in the centraliser of $K$, so indeed $s\cdot k \cdot x = k \cdot s \cdot x = k \cdot x$. Since $K$ is connected, the $K$-action preserves connected components of the fixed set of $S$. \end{proof}

\begin{cor} \label{fixedmfldsofsarekspaces} Each fixed manifold $F$ of the $S$-action is a Hamiltonian $K$-space. In particular, $W(K)$ acts on the set of fixed points $p_i$ of the torus action that are contained in $F$, and on the weights of the torus action on $\tpif$. \end{cor}

Let $p\in M$ be a fixed point of the torus action, and consider the action of $T$ on the tangent space $T_pM$. The tangent space can be broken up into a direct sum of 2-dimensional weight spaces. The point $p$ is contained in some manifold $F$ fixed by the $S$-action. Denote by $NF$ the normal bundle to $F$ in $M$; since the torus action preserves $F$, the decomposition $T_pM=T_pF \oplus N_pF$ respects the weight decomposition. 
 
Recall that an element $\xi$ of a Lie algebra $\kk$ is \emph{regular} if the dimension of its centraliser is minimal among all centralisers of elements of $\kk$. For Lie algebras of compact Lie groups this means that the centraliser of $\xi$ is a maximal torus of $\kk$. Write $\kk^*_{reg}$ for the set of regular elements of $\kk^*$, and $\tt^*_{reg}$ for the intersection of $\kk^*_{reg}$ with $\tt^*$. 

Since the action of $K$ preserves $F$, if $\mu(p)$ is in $\tt^*_{reg}$ then $T_pF$ contains a copy of $\kk/\tt$; in particular the set of weights of the representation on $T_pF$ contains either $\phi$ or $-\phi$ for all positive roots $\phi$ of $K$. Note however that unlike Guillemin and Prato in \cite{gp} we are not assuming all of the $\mu(p)$ are regular; this allows us to consider actions on non-generic coadjoint orbits.  If $\mu(p)$ is not in $\ttr$ then it's fixed by some of the $w_\alpha$, where $w_\alpha \in W(K)$ is the element of the Weyl group that acts on $\tt^*$ by reflection in the hyperplane orthogonal to $\alpha$. These $\alpha$ are then not necessarily weights of the torus action on $T_pF$.

\subsection{Some Notation}

Let $P=\{p_1, \ldots, p_n\}$ be the set of fixed points of the torus action on $M$. For the fixed point $p_i$, let $A_i$ be the set $\{\alpha \in \pk \vert \langle \alpha , \mu(p_i) \rangle =0\}$ of roots $\alpha$ of $K$ such that $\mu(p_i)$ lies on the hyperplane orthogonal to $\alpha$ (and so $w_\alpha$ fixes $\mu(p_i)$). Let $\mathcal{A}_i$ be the group generated by the $w_\alpha$, for $\alpha \in A_i$, and let $U_i$ be the set of left cosets of $\mathcal{A}_i$ in $W(G)$. If $X$ is an $\mathcal{A}_i$-invariant set then we define the action of $U_i$ on $X$ by $[w]\cdot x = w \cdot x$, where $w \in W(G)$, the coset of $\mathcal{A}_i$ containing $w$ is denoted $[w]$, and $x \in X$; this action does not depend on the choice of coset representative and so is well defined. For ease of notation we will often drop the square brackets and refer to elements $w \in U_i$.

 If $w\cdot p_i = p_j$, we will also write $w\cdot i = j$. Let $P^+$ be the set of fixed points which map into the closed positive Weyl chamber $\overline{\mathcal{W}_K}$ under the moment map.
 
Given a fixed point $p_i \in M$, let $F$ be the fixed manifold of the $S$ action containing $p_i$. Let $\{\alpha_{ij}\}^{a_i}_{j=1}$ denote the weights of the torus action on the tangent space to the $K$-orbit of $p_i$ (these form a subset of $\Phi(K)$), let $\{\beta_{ij}\}^{f_i}_{j=1}$ denote the remaining weights of $\tpif$, and let $\{\beta_{ij}\}^{n_i}_{j=f_i+1}$ denote the weights of $(NF)_{p_i}$. We write $B_i = \{\beta_{ij} \vert 1 \leq j \leq n_i \}$. We will polarise the weights as follows. Choose a $\xi \in \tt$ such that $\langle \alpha , \xi \rangle = 0$ for all $\alpha \in W(K)$, and $\langle \lambda , \xi \rangle \neq 0$ for all other weights $\lambda$ of the torus action on tangent spaces to fixed points. So $W(K)$ preserves $\xi$. Note that there may be some weights $\beta_{ij}$ of the torus action on the normal bundle to the $K$-orbit of $p_i$ which also happen to be roots of $K$; such $\beta_{ij}$ will have $\langle \beta_{ij}, \xi \rangle = 0$. Define \emph{polarised} weights

\[ \beta_{ij}^+ = \begin{cases}
    \beta_{ij}, & \beta_{ij}(\xi) \geq 0\\
    -\beta_{ij}, & \beta_{ij}(\xi) < 0
  \end{cases},\]
and for fixed $i$, let $s_i$ be the number of $j$ with $\beta_{ij}(\xi) <0$. Notice that we've chosen $\xi$ such that $s_i = s_{w\cdot i}$ for all $w \in W(K)$.

For the fixed point $p_i$, write $\beta_i = \frac{1}{2}\sum_j \beta_{ij}$. Let $\bb_i$ be the set $\{w \cdot \beta_{ij} \vert w \in W(K), \beta_{ij} \in B_i\}$, and let $\bb^+_i$ be the set of polarised weights $\{\beta^+ \vert \beta \in \bb_i\}$. Let $\overline{\beta_i}=\sum_{\beta \in B_i: \beta = -\beta^+}\beta$. Let $C_i$ be the set of weights $\beta \in \bb_i^+$ that are not equal to $\beta_{ij}^+$ for any $\beta_{ij} \in B_i$; that is, polarisations of all weights in the orbit except for those at the point $p_i$.  Let $m_i(\eta)$ be the multiplicity of $e^\eta$ in $\prod_{\gamma \in C_i}(1-e^{-\gamma})$.

\subsection{The Main Theorem}

\begin{thm} \label{main} Let $G$ be a compact Lie group with maximal torus $T$. Let $(M, \omega, L, \nabla; J)$ be a Hamiltonian $G$-space such that the $T$-action has isolated fixed points. Let $K \subset G$ be a closed connected subgroup with maximal torus $T$.  Write $V^K_\lambda$ for the irreducible $K$-representation of highest weight $\lambda$. Then the character $\chi$ of the quantisation $Q(M, \omega, L, \nabla; J)$ is given by 
\begin{equation}
\label{mainthm}
\chi(\qm) = \sum_{F \in \pi_0(M^S)} \sum_{p_i \in F \cap P^+} \frac{\sum_{\eta \in \Lambda}(-1)^{s_i}m_i(\eta) \chi(V^K_{\mu(p_i)+\overline{\beta_i}+\eta})}{ \prod_{\gamma \in \bb^+_i} (1-e^{-\gamma})}.
\end{equation}
 \end{thm}
 
\begin{proof} The equivariant index theorem (Theorem \ref{eqind}) tells us that $\chi \circ \exp = \int_M e^{c_{1,T}(L)}Td_T(M)$. By localisation with respect to the $T$ action we get 
\begin{equation}
\chi \circ \exp = \sum_{p_i \in M^T} \frac{\bigl(e^{c_{1,T}(L)}Td_T(M)\bigr)\big|_{p_i}}{e_T(T_{p_i}M)}.
\end{equation}
Let $M^S$ denote the fixed set of the $S$ action on $M$. Since every fixed point $p_i \in M^T$ is contained in $M^S$, we can write $M^T$ as the disjoint union over all $F \in \pi_0(M^S)$ of the fixed sets $F^T$ of the $T$ action on $F$. So
\begin{equation}
\chi \circ \exp = \sum_{F \in \pi_0(M^S)} \sum_{p_i \in F^T} \frac{\bigl(e^{c_{1,T}(L)}Td_T(M)\bigr)\big|_{p_i}}{e_T(T_{p_i}M)}.
\end{equation}
At a fixed point $p_i \in F$, we can write $TM\vert_{p_i}$ as $T_{p_i}F \oplus (NF)_{p_i}$. Hence $Td_T(TM)_{p_i} = Td_T(T_{p_i}F)Td_T(NF)_{p_i}$, and $e_T(TM)_{p_i} = e_T(T_{p_i}F)e_T(NF)_{p_i}$. So 
\begin{equation} \label{allbrokenup}
\chi \circ \exp = \sum_{F \in \pi_0(M^S)}\sum_{p_i \in F^T}\frac{e^{(\cit(L))_{p_i}}\tdt(\tpif)\tdt(\npif)}{e_T(\npif)e_T(\tpif)}.
\end{equation}
Recalling the definitions of the equivariant characteristic classes and lemma \ref{wactsonweights}, we can write
\begin{equation}e_T(\tpif) = \prod_j \alpha_{ij}\prod_{k=1}^{f_i} \beta_{ik},
\end{equation}
\begin{equation}
\tdt(\tpif) = \displaystyle \frac{e_T(\tpif)}{\prod_j (1-e^{-\alpha_{ij}})\prod_{k=1}^{f_i} (1-e^{-\beta_{ik}})},
\end{equation}
\begin{equation}
e_T(\npif) =\displaystyle \prod_{l=f_i+1}^{m_i} \beta_{il},
\end{equation}
\begin{equation}
\tdt(\npif)= \frac{\displaystyle \prod_{l=f_i+1}^{m_i} \beta_{il}}{\displaystyle \prod_{l=f_i+1}^{m_i} (1-e^{-\beta_{il}})},
\end{equation}
\begin{equation}
e^{(\cit(L)\vert_F)_{p_i}}=e^{\mu(p_i)}.
\end{equation}
And so our formula becomes 
\begin{equation}\chi \circ \exp = \sum_{F \in \pi_0(M^S)}\sum_{p_i \in F^T} \frac{e^{\mu(p_i)}}{\prod_j (1-e^{-\alpha_{ij}})\prod_{k=1}^{m_i} (1-e^{-\beta_{ik}})}.
\end{equation}

We now organise this formula into characters of representations of $K$. By Lemma \ref{wactsonp}, $W(K)$ acts on the fixed points of the fixed manifolds $F$. This action partitions the fixed points contained in $F$ into $W(K)$-orbits; each orbit contains $\vert U_i \vert$ points, and has a unique representative $p_i$ whose image $\mu(p_i)$ lies in the closed positive Weyl chamber $\overline{\mathcal{W}_K}$. For given $F$, the set of such orbit representatives is $F \cap P^+$. 

Recall from Lemma \ref{wactsonweights} that the action of $w \in W(K)$ takes weights of the tangent space at $p_i$ to weights of the tangent space at $w\cdot p_i$. So we can write our formula more suggestively as 
\begin{equation}
\label{suggestive}
\chi \circ \exp = \sum_{F \in \pi_0(M^S)} \sum_{p_i \in F \cap P^+} \sum_{w \in U_i}\dfrac{\frac{e^{w(\mu(p_i))}}{\prod_j (1-e^{-w\cdot\alpha_{ij}})}}{\prod_{\beta \in B_i} (1-e^{-w\cdot\beta})}.
\end{equation}
Recall that an irreducible $K$-character has the form $\sum_{w \in U_i}\frac{e^{w\cdot \lambda}}{\prod_\alpha(1-e^{-w\cdot \alpha})}$, where the $\alpha$ are the positive roots of $K$ whose inner product with $\lambda$ is non-zero. (When $\lambda$ is regular, this is all the positive roots and this becomes the more familiar $\sum_{w \in W(K)}\frac{e^{w\cdot \lambda}}{\prod_{\phi \in \pk}(1-e^{-w\cdot \phi})}$.) So to get irreducible $K$-characters, we would like to rewrite 
\begin{equation}
\label{rewritethis}
\sum_{w \in U_i}\dfrac{\frac{e^{w(\mu(p_i))}}{\prod_j (1-e^{-w\cdot\alpha_{ij}})}}{\prod_{\beta \in B_i} (1-e^{-w\cdot\beta})}
\end{equation} as 
\begin{equation}\frac{\sum_{w \in U_i}\frac{e^{w\cdot \mu(p_i)}w\cdot R}{\prod_j(1-e^{-w\cdot \alpha_{ij}})}}{H},
\end{equation} where $R = \sum_{\tau \in E}c_\tau e^\tau$, for $E$ a finite subset of the integral weight lattice, and $c_\tau \in \mathbb{Z}$. Equivalently, we wish to rewrite $(\prod_{\beta \in B_i}(1-e^{-\beta}))^{-1}$ as $R/H$ with $H$ required to be $W(K)$-invariant. One $W(K)$-invariant set we have is the set $\bb_i$ of all weights of the torus action on the spaces $T_{w\cdot p_i}M/T_{p_i}(K\cdot p_i)$ for $w$ ranging over $W(K)$. In fact, since we chose $\xi$ to be fixed by $W(K)$, we know $\langle w\cdot \beta, \xi \rangle = \langle \beta , \xi \rangle$ and so $W(K)$ also preserves the set $\bb_i^+$ of polarised weights. Let us therefore take 
\begin{equation}
H = \prod_{\gamma \in \bb^+_i}(1-e^{-\gamma}).
\end{equation}
Thus we find that 
\begin{equation*}
R = \dfrac{\prod_{\gamma \in \bb^+_i}(1-e^{-\gamma})}{\prod_{\beta \in B_i}(1-e^{-\beta})}.
\end{equation*}
Consider $\beta \in B_i$. If $\beta = \beta^+$, then $\beta \in \bb^+_i$, so the corresponding terms cancel. If $\beta = -\beta^+$, then $-\beta \in \bb^+_i$, so we can replace the corresponding terms with a $-e^\beta$. By definition the number of $\beta \in B_i$ with $\beta = -\beta^+$ is $s_i$, and so we are left with 
\begin{equation}
R = (-1)^{s_i}e^{\sum_{\beta \in B_i:\beta=-\beta^+}\beta}\prod_{\gamma \in C_i}(1-e^{-\gamma}),
\end{equation}
where $C_i$ is the set of all polarised weights $\gamma \in \bb_i^+$ of the torus action on the tangent spaces to the whole orbit of fixed points such that neither $\gamma$ nor $-\gamma$ is in $B_i$; that is, neither $\gamma$ nor $-\gamma$ is a weight of the torus action on the tangent space $\tpim$ at the particular fixed point $p_i$. Write $\overline{\beta_i}$ for $\sum_{\beta \in B_i: \beta = -\beta^+}$. So we have rewritten \eqref{rewritethis} as 
\begin{equation}
\label{rewritten}
\frac{\sum_{w \in U_i}w\cdot\left(\frac{e^{\mu(p_i)}(-1)^{s_i}e^{\overline{\beta_i}}\prod_{\gamma \in C_i}(1-e^{-\gamma})}{\prod_j (1-e^{-\alpha_{ij}})}\right)}{\prod_{\gamma \in \bb^+_i}(1-e^{-\gamma})}.
\end{equation}

\begin{lem} \label{trivrepn} Let $A_i \subset \pk$ be the set $\{\alpha \in \pk \vert \langle \alpha, \mu(p_i) \rangle = 0\}$, and let $\mathcal{A}_i$ be the subgroup of $W(K)$ generated by the reflections $w_\alpha$ for $\alpha \in A_i$. Then $$\sum_{w\in \mathcal{A}_i}\frac{1}{\prod_{\alpha \in A_i}(1-e^{-w\cdot \alpha})}=1.$$
\end{lem}
\begin{proof} Let $s_w$ denote the number of $\alpha$ in $A_i$ such that $w\cdot \alpha$ is a negative root of $K$, and let $\rho_{A_i}$be half the sum of the elements of $A_i$. Then $$\sum_{w\in \mathcal{A}_i}\frac{1}{\prod_{\alpha \in A_i}(1-e^{-w\cdot \alpha})}=\frac{\sum_{w\in \mathcal{A}_i}(-1)^{s_w}e^{w\cdot \rho_{A_i}}}{\prod_{\alpha}(e^{\alpha/2}-e^{-\alpha/2})}.$$
We claim that this is equal to $1$. Let $\Psi \subset A_i$ be a set of simple roots (so that each $\alpha \in A_i$ can be expressed uniquely as a non-negative integer combination of the $\psi \in \Psi$). Let $x=(x_1, \ldots, x_{\vert A_i \vert}) \in \{-1,1\}^{\vert A_i\vert }$, and consider the term $X=e^{\sum x_i \alpha_i/2}$ in the expansion of the product in the denominator. 

If for all $i, j, k$ for which $\alpha_i = \alpha_j+\alpha_k$ and $x_j = x_k$ we also have $x_i =x_j$, then $\sum x_i \alpha_i/2 = w\cdot \rho_{A_i}$ for some $w \in \mathcal{A}_i$, and the coefficient of $X$ in the denominator is $\prod x_i = (-1)^{s_w}$. Every $e^{w\cdot \rho_{A_i}}$ appears once in the denominator.  

Otherwise, suppose there are $i, j, k$ with $\alpha_i = \alpha_j + \alpha_k$ and $x_j = x_k$ but $x_i = -x_j$. Then there's another term $\overline{X} = e^{\sum \overline{x_l}\alpha_l/2 }$ where $\overline{x_l}=-x_l$ for $l = i, j, k$ and $\overline{x_l}=x_l$ otherwise. Then $X$ and $\overline{X}$ both appear once in the expansion of the denominator, each with opposite sign and so they cancel. 

Thus we are left with $\prod_\alpha(e^{\alpha/2}-e^{-\alpha/2})=\sum_w (-1)^{s_w}e^{w\cdot \rho_{A_i}}$, which proves the lemma.  
\end{proof}

By multiplying \eqref{rewritten} by $\sum_{w\in \mathcal{A}_i}(\prod_{\alpha \in A_i}(1-e^{-w\cdot \alpha}))^{-1}$, we get
\begin{equation}
\frac{\sum_{w \in W(K)}w\cdot\left(\frac{e^{\mu(p_i)}(-1)^{s_i}e^{\overline{\beta_i}}\prod_{\gamma \in C_i}(1-e^{-\gamma})}{\prod_{\phi \in \pk} (1-e^{-\phi})}\right)}{\prod_{\gamma \in \bb^+_i}(1-e^{-\gamma})},
\end{equation}
which is a quotient of $K$-characters, where the $K$-character of highest weight $\mu(p_i)+\overline{\beta_i}+\lambda$ appears with multiplicity equal to the multiplicity of $e^\lambda$ in $(-1)^{s_i}\prod_{\gamma \in C_i}(1-e^{-\gamma})$. Write $V^K_\lambda$ for the irreducible $K$-representation of highest weight $\lambda$, and $m_i(\eta)$ for the multiplicity of $e^\eta$ in $\prod_{\gamma \in C_i}(1-e^{-\gamma})$. Then summing over all orbits we get 

$$\chi(\qm) = \sum_{F \in \pi_0(M^S)} \sum_{p_i \in F \cap P^+} \frac{\sum_{\eta \in \Lambda}(-1)^{s_i}m_i(\eta) \chi(V^K_{\mu(p_i)+\overline{\beta_i}+\eta})}{ \prod_{\gamma \in \bb^+_i} (1-e^{-\gamma})},$$
which is equation \eqref{mainthm}, as claimed. \end{proof}

\begin{remark} There are at most $2^{\vert C_i \vert}$ points $\eta \in \Lambda$ where $m_i(\eta)$ is non zero.
\end{remark} 

\begin{remark} Although the proof of our main theorem only requires the $T$ action, considering the $S$ action as well helps with the geometric intuition. Consider one connected component $F \in \pi_0(M^S)$ of the fixed set of the $S$-action. By Corollary \ref{fixedmfldsofsarekspaces} $F$ is a Hamiltonian $K$-space, so the character of the quantisation of the restriction of $L$ to $F$ is given by $\int_F e^{c_{1,T}(L)\vert_F}Td_T(F)$. This is an integral of $S$-equivariant forms, so we can apply localisation with respect to the $S$ action on $F$ to get $\sum_{p \in F^T} \frac{e^{c_{1,T}(L)_p}Td_T(T_pF)}{e_T(T_pF)}$. These terms, for each $F \in \pi_0(M^S)$, appear in the formula \eqref{allbrokenup} for the full character $\chi \circ \exp$, each multiplied by $Td_T(NF)_p / e_T(NF)_p$. Geometrically, this tells us that the way a few $K$-characters combine to form a $G$-character is a consequence of the way that the fixed manifolds of the $S$-action -- each of which is a $K$-space -- are embedded in the $G$-space. \end{remark}  

\subsection{Example}
\begin{example}
Let $G = SU(3)$ and let $K\subset G$ be $S(U(2)\times U(1))$. Denote by $\alpha, \beta$ and $\gamma$ the positive roots of $G$, where $\beta = \alpha + \gamma$, and let $\alpha$ be the positive root of $K$. Choose $\xi$ such that the polarised roots are the positive roots. Let $\nu = \beta + \gamma$, and let $M$ be  the coadjoint orbit $G.\nu\cong G/K$. In this case the moment map is the inclusion $i: G\cdot \nu \hookrightarrow \gg^*$ composed with the projection $\pi: \gg^* \rightarrow \tt^*$.

The centraliser of $K$ in $G$ is the subtorus $S$ consisting of matrices of the form \[ \left( \begin{array}{ccc}
e^{i\theta} & 0& 0 \\
0 & e^{i\theta} & 0 \\
0 & 0 & e^{-2i\theta} \end{array} \right),\] for $\theta \in [0,2\pi)$. The $T$ action on $G/K$ has three fixed points: their moment map images are $\nu$, $w\cdot \nu$, and $w^2\cdot \nu$ where $w$ is one of the rotations in $W(G)$. Let's assume $w\cdot \nu$ is in the dominant Weyl chamber $W_K$. These fixed points lie in the two fixed manifolds of the $S$-action: the point $\nu$, and the copy of $K/T\cong \mathbb{P}^1$ which maps to the edge between $w\cdot \nu$ and $w^2\cdot \nu$ under the moment map. Each fixed manifold contains just one $W(K)$-orbit of fixed points. 

For the fixed manifold $F \cong \mathbb{P}^1$, the weights of tangent spaces are $-\gamma$ (at $T_{w\cdot p}G/K$) and $-\beta$ (at $T_{w^2\cdot p}G/K$). So the denominator will be $(1-e^{-\beta})(1-e^{-\gamma})$. Each fixed point in this $W(K)$-orbit has one negative weight on its tangent space that's not a root of $K$, so $s_i = 1$. Here $\overline{\beta_i}$ is the sum of the negative weights at $w\cdot \gamma$, which is $-\gamma$. Recall that $C_i$ is the set of weights $\delta$ of the $T$ action on tangent spaces to $M$ at fixed points of the orbit of $p_i$ such that neither $\delta$ nor $-\delta$ is a weight of the tangent space at $p_i$, where $p_i$ is the fixed point in the dominant Weyl chamber for $K$; in this case $p_i$ is $w\cdot \nu$ and $C_i$ is $-\beta$. Thus $m_i(\eta)$ is the multiplicity of $e^\eta$ in $1-e^{-\beta}$, so 

\[ m_i(\eta) = \begin{cases}
    1, & \eta = 0\\
    -1, & \eta = -\beta\\
    0, & \text{otherwise.}
  \end{cases}\]
At the point $p$, the weights of the tangent space are $\beta$ and $\gamma$, so the denominator is again $(1-e^{-\beta})(1-e^{-\gamma})$. Both weights are positive so $s_i=0$ and $\overline{\beta_i}=0$. This time there are no weights on tangent spaces in that orbit that are not weights of the tangent space at that point, so $C_i$ is empty and $m_i(\eta)$ is non-zero (and equal to one) only when $\eta = 0$. 

So our theorem tells us that 

$$\chi(V^G_\nu) = \frac{\chi(V^K_{\nu})-\left(\chi(V^K_{w\cdot\nu - \gamma})-\chi(V^K_{w\cdot\nu -\beta -\gamma})\right)}{(1-e^{-\beta})(1-e^{-\gamma})}.$$

\end{example}

\subsection{Relation to the GKRS formula}

In their paper \cite{gkrs}, Gross, Kostant, Ramond and Sternberg consider equal rank Lie algebras $\kk \subset \gg$ with $\gg$ semisimple and $\kk$ reductive, and give a formula for the characters of irreducible $\gg$-representations in terms of certain irreducible $\kk$-characters. Since representations of Lie groups are also representations of their Lie algebras, in the case where $G$ is semisimple and the quantisation of the $G$-action on $M$ is an irreducible $G$-representation, we expect our formula to agree with the GKRS formula. We now show that the two formulae do indeed coincide.

At the end of the proof of Theorem \ref{main}, we had expressed the contribution from one orbit of fixed points as 

\begin{equation}\label{toconvert}
\dfrac{\sum_{w \in W(K)}w\cdot \left(\frac{e^{\mu(p_i)}(-1)^{s_i}e^{\overline{\beta_i}}\prod_{\gamma \in C_i}(1-e^{-\gamma})}{\prod_{\phi \in \pk}(1-e^{-\phi})}\right)}{\prod_{\gamma \in \bb^+_i}(1-e^{-\gamma})}.
\end{equation}
We may write $\bb^+_i = \{\beta \in B_i \vert \beta = \beta^+ \} \sqcup \{-\beta \vert \beta \in B_i, \beta = -\beta^+ \}\sqcup C_i$. So multiplying numerator and denominator of \eqref{toconvert} by $\prod_{\gamma \in \bb^+_i}e^{\gamma/2}$ (which is $W(K)$-invariant), we get 

\begin{equation}\label{withhalves}
\dfrac{\sum_{w\in W(K)}w\cdot \left( \frac{e^{\mu(p_i)}(-1)^{s_i}e^{\beta_i}\prod_{\gamma \in C_i}(e^{\gamma/2}-e^{-\gamma/2})}{\prod_{\phi \in \pk}(1-e^{-\phi})} \right)}{\prod_{\gamma\in \bb^+_i}(e^{\gamma/2}-e^{-\gamma/2})}.
\end{equation}
Since $\gamma/2$ is not necessarily a weight of $T$, the numerator of \eqref{withhalves} is a character not of $K$ but of a covering of $K$; it is however a character of the Lie algebra $\kk$. Thus if we work at the Lie algebra level we have the alternative expression 
\begin{equation} \label{liealgversion}
\chi(\qm) =\sum_{F \in \pi_0(M^S)} \sum_{p_i \in F \cap P^+} \dfrac{\sum_{\eta \in \Lambda}(-1)^{s_i}\tilde{m_i}(\eta)\chi(V^\kk_{\mu(p_i)+\beta_i+\eta})}{\prod_{\gamma \in \bb^+_i}(e^{\gamma/2}-e^{-\gamma/2})}, 
\end{equation}
where $\tilde{m_i}(\lambda)$ is the multiplicity of $e^\lambda$ in $\prod_{
\gamma \in C_i}(e^{\gamma/2}-e^{-\gamma/2})$. 

Let $K \subset G$ be compact connected Lie groups of equal rank, and let $G$ act on $G/T$ (by left multiplication). Choose positive roots for $G$, let $\rho_G$ be half their sum, and let $\rho_K$ be half the sum of the positive roots of $K$. Recall that the choice of positive roots determines a complex structure on $G/T$; let $J$ be the associated almost complex structure. Choose $\xi$ such that the polarised roots are the positive roots. Let $\lambda$ be a dominant integral weight, and let $\ll = G \times_T \cc_{(\lambda)} \to G/T$ be the prequantisation line bundle, where $T$ acts on $\cc_{(\lambda)}$ with weight $\lambda$ (as in \cite{bott}). In this case the quantisation is the irreducible representation of $G$ of highest weight $\lambda$. The images of the fixed points of the $T$ action are the $w\cdot \lambda$, for $w \in W(G)$. Write $p_w$ for the fixed point whose image $\mu(p_w)$ is $w\cdot \lambda$. Let $\mathcal{W}_G$, $\mathcal{W}_K$ denote the positive Weyl chambers of $G$, $K$ respectively (so $\mathcal{W}_G \subset \mathcal{W}_K$). There is one fixed manifold of the $S$ action containing each of the fixed points whose images lie in the positive Weyl chamber $\mathcal{W}_K$, and each of those fixed manifolds contains a single $W(K)$-orbit of fixed points. As in Gross, Kostant, Ramond and Sternberg's paper \cite{gkrs}, let $C$ be the subset of $W(G)$ mapping $\mathcal{W}_G$ into $\mathcal{W}_K$; the fixed manifolds are indexed by $C$. The weights of the torus action on $T_{p_w}M$ are $\{w\cdot \phi \vert \phi \in \pg \}$. The orbit $K\cdot \lambda$ is a generic coadjoint orbit (of $K$) and so the set of weights of the $T$ action on the tangent space to the orbit is $\{w\cdot \phi \vert \phi \in \pg\}\cap \Phi(K)$, while the remaining weights make up the $\beta_w$. For every fixed manifold $F$, $\bb^+_i$ is exactly the set $\pgk$ of positive roots of $G$ that are not roots of $K$. For each weight $\beta\in B_w$ of the torus action on the tangent space to a fixed point, either $\beta$ or $-\beta$ appears as a weight of the torus action on the tangent space at each other fixed point in the $W(K)$-orbit; thus $C_i$ is empty, and so $\tilde{m}_i(\eta)$ is non-zero only when $\eta = 0$ in which case $\tilde{m}_i(\eta) = 1$. By definition $s_c$ is the number of positive roots changed into negative roots under $c$; that is, $(-1)^{s_c} = \epsilon(c)$. The weights at $T_{p_c}G/T$ are $\{c\cdot\phi \vert \phi \in \Phi^+(G)\}$. Note that if $c\cdot\phi \in \Phi(K)$ then $c\cdot\phi \in \Phi^+(K)$ since $C$ maps $\mathcal{W}_G$ into $\mathcal{W}_K$. So $\beta_c = \frac{\sum_{\phi \in \Phi^+(G)}c\cdot\phi - \sum_{\alpha \in \Phi^+(K)}\alpha}{2} = c\cdot\rho_G - \rho_K$. Substituting this all into equation \eqref{liealgversion}, we get

$$\chi(V) = \frac{\sum_{c \in C}\epsilon(c)\chi(V_{c(\lambda + \rho_G)-\rho_K})}{\prod_{\phi \in \pgk}(e^{\phi/2}-e^{-\phi/2})},$$ which is exactly the GKRS formula (equation 5 in \cite{gkrs}).

An alternative approach to obtaining a GKRS-like formula that applies to Lie groups is discussed in Landwebber and Sjamaar's paper \cite{ls}.

\section{Multiplicity Formula}

In this section, we will give a formula for the multiplicity in $\qm$ of $V^K_\lambda$,the irreducible representation of $K$ of highest weight $\lambda$. We will compare our formula with the Guillemin-Prato formula from \cite{gp}. 

\subsection{Derivation of the Multiplicity Formula}

Our starting point is the equation 

\begin{equation}
\label{currentsuggestive}
\chi \circ \exp = \sum_{F \in \pi_0(M^S)} \sum_{p_i \in F \cap P^+} \sum_{w \in U_i}\dfrac{\frac{e^{w(\mu(p_i))}}{\prod_j (1-e^{-w\cdot\alpha_{ij}})}}{\prod_{\beta \in B_i} (1-e^{-w\cdot\beta})}.
\end{equation}
from the previous section. Let us start by only considering fixed points belonging to one Weyl orbit. The contribution to formula \eqref{currentsuggestive} from such a set of points is 

\begin{equation}
\label{oneorbit}
\sum_{w \in U_i} \frac{\frac{e^{w\cdot\mu(p_i)}}{\prod_j (1-e^{-w\cdot\alpha_{ij}})}}{\prod_k(1-e^{-w\cdot\beta_{ik}})}.
\end{equation}

We polarise the $\beta_{ij}$ as above. If $\beta_{ij}=\beta_{ij}^+$, we have $(1-e^{-\beta_{ik}})^{-1}=(1-e^{-\beta_{ik}^+})^{-1}$; if $\beta_{ij}=-\beta_{ij}^+$ then rewrite $(1-e^{-\beta_{ik}})^{-1}$ as $(-e^{-\beta^+_{ik})}(1-e^{-\beta^+_{ik}})^{-1}$. By definition the number of $\beta_{ij}$ with $\beta_{ij}=-\beta_{ij}^+$ is $s_i$; thus $(\prod_k(1-e^{-\beta_{ik}}))^{-1}$ becomes $((-1)^{s_i}e^{-(\beta^+_i - \beta_i)})(\prod_k (1-e^{-\beta^+_{ik}}))^{-1}$.
As discussed earlier, we have $s_{w(i)}=s_i$, and $w\cdot(\beta^+_i - \beta_i) = \beta^+_{w(i)} - \beta_{w(i)}$. Thus we can rewrite \eqref{oneorbit} as 

\begin{equation}
\sum_{w \in U_i} \frac{\frac{(-1)^{s_i}w\cdot(e^{\mu(p_i)-\beta^+_i + \beta_i})}{\prod_j (1-e^{-w\cdot\alpha_{ij}})}}{\prod_k(1-e^{-w\cdot\beta^+_{ik}})}
\end{equation}
Expanding $(1-x)^{-1}$ as $1+x+x^2+\cdots$, this gives us

\begin{equation}
\sum_{w\in U_i}\frac{(-1)^{s_i}w\cdot(e^{\mu(p_i)-\beta^+_i + \beta_i})}{\prod_j (1-e^{-w\cdot\alpha_{ij}})}\prod_k(1+e^{-w\cdot \beta^+_{ik}}+e^{-2w\cdot \beta^+_{ik}}+e^{-3w\cdot \beta^+_{ik}}+\cdots).
\end{equation}
For $\zeta \in \tt^*$, write $P_i(\zeta)$ for the number of ways to write $\zeta$ as a sum $\sum_jc_j\beta^+_{ij}$, where the $c_j$ are non-negative integers. So the contribution of this $W(K)$-orbit of fixed points is 

\begin{equation} \label{orbitcont}
(-1)^{s_i}\sum_\zeta P_i(\zeta)\sum_{w \in U_i}\frac{w\cdot(e^{\mu(p_i)-\beta^+_i + \bi + \zeta})}{\prod_j (1-e^{-w\cdot\alpha_{ij}})}.
\end{equation}
Recall that an irreducible character of $K$ of highest weight $\lambda$ takes the form $$\sum_{w \in W(K)}\frac{e^{w\cdot \lambda}}{\displaystyle \prod_{\alpha \in \pk}(1-e^{-w\cdot\alpha})}.$$ 
Note that if $\mu(p_i)$ is regular and $w^{-1}\mu(p_i)$ is in the positive Weyl chamber $W_K$, then the $\alpha_{ij}$ are exactly the $w\cdot \alpha$, for $\alpha \in \pk$. If $\mu(p_i)$ is not regular and lies on a hyperplane orthogonal to the root $\alpha$ of $K$, then neither $\alpha$ nor $-\alpha$ will be a weight on $T_{p_i}(K\cdot p_i)$. As before let $A_i$ be the set $\{\alpha \in \pk \vert \langle \alpha, \mu(p_i) \rangle = 0\}$; notice that this is exactly the set of positive roots $\alpha$ of $K$ such that neither $\alpha$ nor $-\alpha$ is a weight of the torus action on $T_{p_i}(K\cdot p_i)$. Hence $\pk = \{\alpha_{ij}\} \sqcup A_i$. By lemma \ref{wactsonweights} we know that $A_{w\cdot i}=\{w\cdot \alpha \vert \alpha \in A_i\}$ for $w \in W(K)$. Let $\mathcal{A}_i < W(K)$ be the subgroup generated by the reflections $w_\alpha$ for $\alpha \in A_i$. Note that $\prod_{\alpha \in A_i}((1-e^{-\alpha})^{-1}+(1-e^\alpha)^{-1})$ is $\mathcal{A}_i$-invariant, so the action of $U_i$ on this product is well defined. 

Using Lemma \ref{trivrepn} we can rewrite \eqref{orbitcont} as 
\begin{equation} \label{wv}
(-1)^{s_i}\sum_\zeta P_i(\zeta)  \sum_{w \in U_i}\left(\frac{w\cdot(e^{\mu(p_i)-\beta^+_i + \bi - \zeta})}{\prod_k (1-e^{-w\cdot\alpha_{ik}})} \sum_{v\in \mathcal{A}_i}\frac{1}{\prod_{\alpha \in A_i}(1-e^{-v\cdot \alpha})} \right).
\end{equation}
But since $\pk = \{\alpha_{ik}\}\sqcup A_i$, this can be simplified to  
\begin{equation} \label{wholewk}
(-1)^{s_i}\sum_\nu P_i(\nu) \sum_{w \in W(K)}\frac{w\cdot(e^{\mu(p_i)-\beta^+_i + \bi - \nu})}{\prod_{\alpha \in \pk} (1-e^{-w\cdot\alpha})}.
\end{equation}
\begin{remark}
In the previous formula \eqref{wv}, the sum over $U_i$ corresponded to a sum over the fixed points in one $W(K)$-orbit. In the formula \eqref{wholewk}, we instead sum over the set of pairs $(p,w)$ where $p = w\cdot p_i$ for $w\in W(K)$.  Geometrically, what we have done is replace each point $p$ with several copies of itself, one for each Weyl chamber on whose boundary $\mu(p)$ lies.
\end{remark} Let $Z_i$ denote the set of pairs $(p,w)$ where $p = w\cdot p_i$ for $w\in W(K)$. Note that $\vert Z_i \vert = \vert W(K) \vert$.
 We let $Z$ be the union over all $p_i \in P^+$ of the $Z_i$; that is, let $Z$ be the set of pairs $(p,w)$ where $p$ is a fixed point of the $T$ action on $M$ and $w\in W(K)$ is such that $w^{-1}(\mu(p))$ is in the closed positive Weyl chamber $\overline{\mathcal{W}_K}$. (So each regular point appears once; the non-regular points appear once for each Weyl chamber in whose boundaries they lie.) 

Equation \eqref{wholewk} expresses the contribution of this orbit as a sum of irreducible $K$-characters; we would now like to know the multiplicity with which the $K$-character of highest weight $\lambda$ occurs, for each $\lambda \in \tt^*$. 
Let $\rho_K$ denote half the sum of the positive roots of $K$. We can rewrite \eqref{wholewk} as 
\begin{equation} \label{qi}
(-1)^{s_i}\sum_\zeta P_i(\zeta)\sum_{(p_i,w)\in Z_i}\epsilon(w) \frac{e^{\mu(p_i)-\beta^+_i+\beta_i-\zeta-\rho_K+w\cdot \rho_K}}{\prod_{\alpha \in \pk}(1-e^{-\alpha})}
\end{equation}
The contribution of one $W(K)$-orbit to the multiplicity of $\lambda$ is given by the multiplicity in \eqref{qi} of terms that look like $\frac{e^\lambda}{\prod_{\alpha \in \pk}(1-e^{-\alpha})}$. So write $\lambda = \mu(p_i)-\beta^+_i+\beta_i-\zeta -\rho_K+w\cdot\rho_K$. Then we can rewrite the contribution from one $W(K)$ orbit as
\begin{equation}
(-1)^{s_i}\sum_\lambda \sum_{(p_i,w)\in Z_i} \epsilon(w) P_i(-\lambda + \mu(p_i)-\beta^+_i +\beta_i-\rho_K +w\cdot\rho_K)\frac{e^\lambda}{\prod_{\alpha \in \pk}(1-e^{-\alpha})}
\end{equation}
and hence, summing over all Weyl orbits, we get the following formula:
\begin{thm} Let $K \subset G$ be compact Lie groups of equal rank with a common maximal torus $T$. Let $(M, \omega, L, \nabla; J)$ be a compact Hamiltonian $G$-space, and let $\qm$ be its quantisation. Let $\lambda$ be a dominant weight for $K$. Then the multiplicity of the irreducible $K$-representation of highest weight $\lambda$ in the $G$-representation $\qm$ is given by $$\#(\lambda, \qm) = \sum_{(p_i,w)\in Z} (-1)^{s_i}\epsilon(w)P_i(\mu(p_i)+w\cdot\rho_K +\bi -\lambda - \rho_K - \beta^+_i).$$ \end{thm}
\begin{remark} This is almost the Guillemin-Prato multiplicity formula from \cite{gp}, generalised to allow fixed points whose images under $\mu$ are not regular. The two formulae do not quite agree; the formula appearing in \cite{gp}, in our notation, is the following: $$\#(\lambda, \qm) = \sum_{(p_i,w)\in Z} (-1)^{s_i}\epsilon(w)P_i(\lambda +w\cdot\rho_K +\beta^+_i -\bi - \rho_K -\mu(p_i)).$$ \end{remark}

\subsection{Examples}
We will illustrate our formula using the same examples as in the previous section. 

\begin{example}[Generic coadjoint orbits] Suppose $M=G/T$, and denote by $\nu$ the moment map image of the fixed point $p$ that lands in the dominant Weyl chamber. So we can identify $G/T$ with the coadjoint orbit $G\cdot \nu$. Then the set of fixed points is $\{u\cdot p \vert u \in W(G)\}$, and their moment map images are the $u \cdot \nu$. Write $p_u = u\cdot p$ for $u\in W(G)$. So the set $Z$ is $\{(p_u, w)\vert u \in W(G), w \in W(K), w^{-1}(u\cdot \nu) \in \mathcal{W}_K\}$. The weights of the $T$ action on the spaces $T_pM/T_p(K\cdot p)$ are the roots of $G$ that are not roots of $K$. So $\{\beta_{uj}\}=\pgk$, and at the point $p_u$ the $\beta_{uj}$ are $\{u\cdot \phi \vert \phi \in \pgk, u\cdot \phi \notin \Phi(K)\}$. For each $u$, $\beta_u = u\cdot \rho_G - w \cdot \rho_K$, and $\beta_u^+=\rho_G - \rho_K$.  Since the $\beta_{uj}^+$ are the same at each fixed point, we will write $P(\lambda)$ for the number of ways to express $\lambda$ as a sum $\sum_{j}c_j\beta_j$, where the $\beta_j$ are the positive roots of $G$ that are not roots of $K$, and the $c_j$ are non-negative integers. Finally, $(-1)^{s_u}\epsilon(w)=(-1)^n$, where $n$ is the number of roots of $G$ that change sign under $u$, and so $(-1)^{s_u}\epsilon(w) = \epsilon(u)$. So in this case, our formula becomes 
\begin{equation} \label{coadj}
\#(\lambda, \qm) = \sum_{u \in W(G)}\epsilon(u)P(u\cdot \nu+u\cdot\rho_G-\lambda-\rho_G).
\end{equation}
\begin{remark}
Goodman and Wallach proved a branching formula which gives the multiplicity of the irreducible representation of $H$ of highest weight $\lambda$ in the irreducible representation of $G$ of highest weight $\nu$ (Theorem 8.2.1 in \cite{gw}). They do not require $G$ and $H$ to be of equal rank, only that the chosen maximal torus of $H$ is contained in that of $G$, but in the case where the two groups do have a common maximal torus their formula coincides with our \eqref{coadj} above.
\end{remark}
\end{example}

\begin{example}

Consider the same non-generic coadjoint orbit we discussed in the previous section. That is, take $G=SU(3)$ and $K=S(U(2)\times U(1))$. Let $\alpha, \beta$ and $\gamma$ be the positive roots of $G$ (with $\alpha + \gamma = \beta$), and take $\alpha$ to be the positive root of $K$. Let $\nu = \beta + \gamma$, and let $M$ be the coadjoint orbit $G\cdot \nu \cong G/K$. Then the torus action has three fixed points: $\nu$, $w_+\nu$ and $w_-\nu$ (where $w_+$ and $w_-$ are the rotations in $W(G)$ that map $\nu$ into the chambers $\mathcal{W}_K$ and $-\mathcal{W}_K$ respectively). 
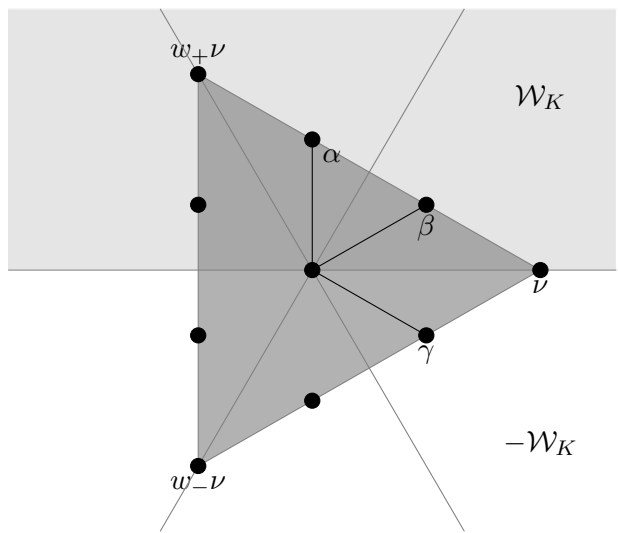
\begin{center}
\begin{tikzpicture}
\draw [white, fill=gray, fill opacity =0.2](-4,0) --(4,0) --(4,3.46410) --(-4,3.46410);
\node [above] at (3,2) {$\mathcal{W}_K$};
\node [below] at (3,-2) {$-\mathcal{W}_K$};
\draw [gray, fill=gray, fill opacity =0.6](-1.5,2.598) --(3,0) --(-1.5,-2.598) --(-1.5,2.598);
\node [above] at (-1.5,2.598) {$w_+ \nu$};
\node [below] at (-1.5,-2.598) {$w_- \nu$};
\node [below] at (3,0) {$\nu$};
\draw [gray](-4,0) --(4,0);
\draw [gray, rotate=60](-4,0) --(4,0);
\draw [gray, rotate=120](-4,0) --(4,0);
\draw [->][black](0,0) --(0,1.732);
\node [below right] at (0,1.732) {$\alpha$};
\draw [->][black, rotate=-60](0,0) --(0,1.732);
\node [below] at (1.5,0.866025) {$\beta$};
\draw [->][black, rotate=-120](0,0) --(0,1.732);
\node [below] at (1.5,-0.866025) {$\gamma$};
\draw [fill] (3,0) circle [radius=0.1];
\draw [fill] (-1.5,2.598) circle [radius=0.1];
\draw [fill] (-1.5,-2.598) circle [radius=0.1];
\draw [fill] (0,1.732) circle [radius=0.1];
\draw [fill] (1.5,0.866025) circle [radius=0.1];
\draw [fill] (0,0) circle [radius=0.1];
\draw [fill] (1.5,-0.866025) circle [radius=0.1];
\draw [fill] (0,-1.732) circle [radius=0.1];
\draw [fill] (-1.5,0.866025) circle [radius=0.1];
\draw [fill] (-1.5,-0.866025) circle [radius=0.1];
\end{tikzpicture}
\end{center}
Let $u$ be the non trivial element in $W(K)$. Since $w_+\nu$ is in the positive chamber $W_K$, $w_-\nu$ is in the other chamber and $\nu$ is on the boundary wall, the set $Z$ consists of the four elements $(w_+\nu, e), (w_-\nu, u), (\nu, e),$ and $(\nu, u)$. The data at each of these points is as follows:
\begin{center}
\begin{tabular}{|c|c|c|c|c|c|}
\hline
$(p,w)\in Z$ & Weights of $T$ action on $T_pM/T_p(K\cdot p)$& $s_i$ & $\epsilon(w)$  & $\beta_i$ & $\beta_i^+$ \\
\hline
$(w_+\nu,e)$ & $-\gamma$ & 1 & 1 & $-\gamma/2$ & $\gamma/2$ \\
$(w_-\nu,u)$ & $-\beta$ & 1 & -1 & $-\beta/2$ & $\beta/2$ \\
$(\nu,e)$ & $\beta, \gamma$ & 0 & 1 & $(\beta + \gamma)/2$ & $(\beta +\gamma)/2$ \\
$(\nu,u)$ & $\beta, \gamma$ & 0 & -1 & $(\beta + \gamma)/2$ & $(\beta +\gamma)/2$ \\
\hline
\end{tabular}
\end{center}
In this case $\rho_K=\alpha/2$. Putting this into our formula, we get $$-P_{\gamma}(w_+\nu-\lambda-\gamma)+P_{\beta}(w_-\nu - \beta - \alpha-\lambda)+P_{\beta,\gamma}(\nu-\lambda)-P_{\beta,\gamma}(\nu - \alpha - \lambda),$$ where $P_A(\psi)$ denotes the number of ways to write $\psi$ as a sum $\sum_{a \in A}c_a a$ where the $c_a$ are in $\mathbb{N}$. Notice that in this example, each $A$ is a linearly independent set of roots, so $P_A$ will always be one or zero. For each point $(p_i,w)\in Z$, we will shade the region of $\tt^*$ containing those $\lambda$ where $P_i(\lambda + \beta^w_i - \beta_i + \rho_K - w\cdot\rho_K - \mu(p_i))=1$ with red lines of positive slope if $(-1)^{s_i}\epsilon(w)$ is $+1$ or blue lines of negative slope if $(-1)^{s_i}\epsilon(w)$ is $-1$. (Thus the regions where cancellations occur will appear hashed in red and blue.) The point $(w_+\nu,e)$ contributes a multiplicity of $-1$ for every $\lambda$ that can be written as $w_+\nu -\gamma - c\gamma$ for $c \in \mathbb{N}$, so we illustrate this with a blue dashed halfline starting at $w_+\nu - \gamma$ and travelling in the $-\gamma$ direction. The point $(\nu, e)$ contributes a multiplicity of $+1$ for every $\lambda$ that can be written as $\nu -b\beta -c\gamma$ for $b, c \in \mathbb{N}$, so this is illustrated with a red shaded cone with vertex $\nu$. The contributions from the other points are shaded in the same way.
\begin{figure}
\begin{center}
\begin{tikzpicture}
\draw [fill] (3/2,0) circle [radius=0.1];
\draw [black] (3/2,0) circle [radius=0.2];
\draw [fill] (-1.5/2,2.598/2) circle [radius=0.1];
\draw [black] (-1.5/2,2.598/2) circle [radius=0.2];
\draw [fill] (-1.5/2,-2.598/2) circle [radius=0.1];
\draw [fill] (0,1.732/2) circle [radius=0.1];
\draw [black] (0,1.732/2) circle [radius=0.2];
\draw [fill] (1.5/2,0.866025/2) circle [radius=0.1];
\draw [black] (1.5/2,0.866025/2) circle [radius=0.2];
\draw [fill] (0,0) circle [radius=0.1];
\draw [fill] (1.5/2,-0.866025/2) circle [radius=0.1];
\draw [fill] (0,-1.732/2) circle [radius=0.1];
\draw [fill] (-1.5/2,0.866025/2) circle [radius=0.1];
\draw [fill] (-1.5/2,-0.866025/2) circle [radius=0.1];
\draw [fill] (-3/2,0) circle [radius=0.1];
\draw [fill] (-3/2,-1.732/2) circle [radius=0.1];
\draw [fill] (-3/2,1.732/2) circle [radius=0.1];
\draw [fill] (-3/2,3.4641/2) circle [radius=0.1];
\draw [fill] (-3/2,-3.4641/2) circle [radius=0.1];
\draw [orange](-1.5/2,2.598/2) --(3/2,0) --(-1.5/2,-2.598/2) --(-1.5/2,2.598/2);
\node [above] at (-1.5/2,2.598/2) {$w_+ \nu$};
\node [below] at (-1.5/2,-2.598/2) {$w_- \nu$};
\node [below] at (3/2,0) {$\nu$};
\draw [red, pattern color=red, pattern = north east lines](-6/2,5.19615/2) --(3/2,0) --(-6/2,-5.19615/2) --(-6/2,5.19615/2);
\draw [blue, pattern color=blue, pattern = north west lines](-6/2,3.461/2) --(3/2,-1.73205/2) --(-6/2,-6.928203/2) --(-6/2,3.461/2);
\draw [blue, dashed, ultra thick](-3/2,3.46410/2) --(-6/2,5.19615/2);
\draw [red, dashed, ultra thick](-3/2,-3.46410/2) --(-6/2,-5.19615/2);
\draw [white, fill=white, fill opacity =0.6](-6/2,0) --(4/2,0) --(4/2,-3.46410/2) --(-6/2,-6.928203/2);
\draw [yellow, ultra thick](-7,0) --(4,0);

\end{tikzpicture}
\caption{The circled points are the only points in $\overline{\mathcal{W}_K}$ with a non-zero net multiplicity, showing that they are the highest weights of the four characters of $K$ that appear in $\qm$ with multiplicity one.}
\label{colour}
\end{center}
\end{figure}
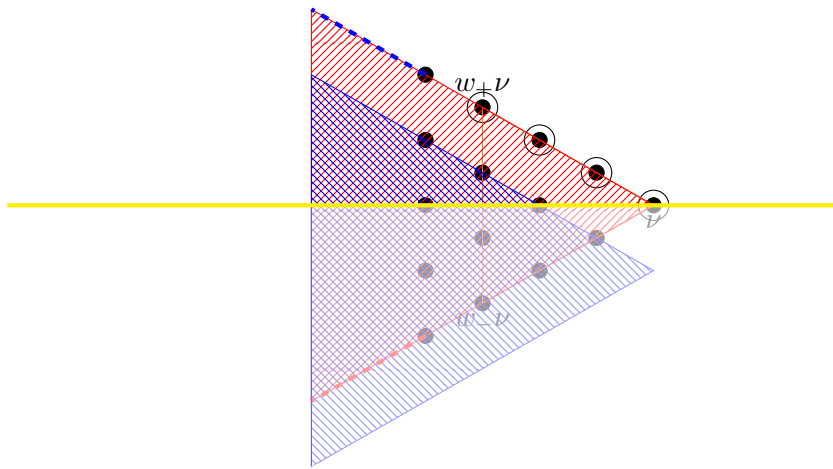

The portion of the picture we are interested in is the top half, corresponding to the closed positive Weyl chamber $\overline{\mathcal{W}_K}$, since our formula tells us, for an element $\lambda \in W_K$, the multiplicity of the irreducible $K$ representation of highest weight $\lambda$. We can see that most of this region is either unshaded or hashed in red and blue; in fact the only integral weights in $\mathcal{W}_K$ that lie in the positive region are the four lying along the top edge of the triangle. Thus we see that our $G$ character breaks up into the four $K$ characters which have the circled weights as their highest weights (see figure~\ref{colour}).        
\end{example}

\bibliographystyle{plain}

\bibliography{refs}

\end{document}